\documentclass[a4paper,12pt]{article}
\usepackage{tikz}
\usetikzlibrary{cd}
\usepackage[T1]{fontenc}
\usepackage{color}

\usepackage{comment}
\usepackage{enumerate}

\usepackage{latexsym}
\usepackage{amscd}
\usepackage{amsmath}
\usepackage{amssymb}
\usepackage{mathrsfs}

\usepackage{array}
\usepackage{amsthm}
\usepackage{float}
\usepackage{graphicx}

\usepackage{booktabs}
\usepackage[normalem]{ulem}

\usepackage{framed}
\usepackage{comment}
\definecolor{shadecolor}{gray}{0.875}

\theoremstyle{plain}
\newtheorem{theorem}{Theorem}[section]
\newtheorem{corollary}[theorem]{Corollary}
\newtheorem{lemma}[theorem]{Lemma}
\newtheorem{definition-lemma}[theorem]{Definition-Lemma}

\newtheorem{proposition}[theorem]{Proposition}
\newtheorem{conjecture}[theorem]{Conjecture}

\newtheorem{step}{Step}

\theoremstyle{definition}
\newtheorem{definition}[theorem]{Definition}
\newtheorem{remark}[theorem]{Remark}

\numberwithin{equation}{section}
\numberwithin{figure}{section}
\numberwithin{table}{section}

\usepackage{prettyref}

\newrefformat{th}{Theorem~\ref{#1}}
\newrefformat{cr}{Corollary~\ref{#1}}
\newrefformat{lm}{Lemma~\ref{#1}}
\newrefformat{dl}{Definition-Lemma~\ref{#1}}
\newrefformat{df}{Definition~\ref{#1}}
\newrefformat{cl}{Claim~\ref{#1}}
\newrefformat{sl}{Sublemma~\ref{#1}}
\newrefformat{pr}{Proposition~\ref{#1}}
\newrefformat{cj}{Conjecture~\ref{#1}}
\newrefformat{st}{Step~\ref{#1}}
\newrefformat{sc}{Section~\ref{#1}}
\newrefformat{df}{Definition~\ref{#1}}
\newrefformat{rm}{Remark~\ref{#1}}
\newrefformat{q}{Question~\ref{#1}}
\newrefformat{pb}{Problem~\ref{#1}}
\newrefformat{cd}{Condition~\ref{#1}}
\newrefformat{eg}{Example~\ref{#1}}
\newrefformat{he}{Heore~\ref{#1}}
\newrefformat{fg}{Figure~\ref{#1}}
\newrefformat{tb}{Table~\ref{#1}}
\newrefformat{as}{Assumption~\ref{#1}}

\newdimen\argwidth
\def\db[#1\db]{%
 \setbox0=\hbox{$#1$}\argwidth=\wd0
 \setbox0=\hbox{$\left[\box0\right]$}
  \advance\argwidth by -\wd0
 \left[\kern.3\argwidth\box0 \kern.3\argwidth\right]}

\newcommand{\rank}{\operatorname{rank}}
\newcommand{\diag}{\operatorname{diag}}

\renewcommand{\det}{\operatorname{det}}

\newcommand{\Image}{\operatorname{Im}}
\newcommand{\coker}{\operatorname{coker}}

\newcommand{\GL}{\operatorname{GL}}
\newcommand{\PGL}{\operatorname{PGL}}
\newcommand{\SL}{\operatorname{SL}}

\newcommand{\coh}{\operatorname{coh}}

\newcommand{\Pic}{\operatorname{Pic}}
\newcommand{\Ext}{\operatorname{Ext}}

\newcommand{\Hom}{\operatorname{Hom}}

\newcommand{\Spec}{\operatorname{Spec}}

\newcommand{\Amp}{\operatorname{Amp}}

\newcommand{\bC}{\ensuremath{\mathbb{C}}}

\newcommand{\bL}{\ensuremath{\mathbb{L}}}

\newcommand{\bP}{\ensuremath{\mathbb{P}}}
\newcommand{\bQ}{\ensuremath{\mathbb{Q}}}
\newcommand{\bR}{\ensuremath{\mathbb{R}}}

\newcommand{\bZ}{\ensuremath{\mathbb{Z}}}

\newcommand{\scE}{\ensuremath{\mathcal{E}}}
\newcommand{\scF}{\ensuremath{\mathcal{F}}}

\newcommand{\scL}{\ensuremath{\mathcal{L}}}
\newcommand{\scM}{\ensuremath{\mathcal{M}}}

\newcommand{\scO}{\ensuremath{\mathcal{O}}}

\newcommand{\scR}{\ensuremath{\mathcal{R}}}

\newcommand{\scU}{\ensuremath{\mathcal{U}}}

\newcommand{\cE}{\ensuremath{\mathcal{E}}}

\newcommand{\cL}{\ensuremath{\mathcal{L}}}
\newcommand{\cM}{\ensuremath{\mathcal{M}}}

\newcommand{\cO}{\ensuremath{\mathcal{O}}}

\newcommand{\cU}{\ensuremath{\mathcal{U}}}

\newcommand{\Btilde}{{\widetilde{B}}}

\newcommand{\Gtilde}{{\widetilde{G}}}

\newcommand{\Ttilde}{{\widetilde{T}}}

\newcommand{\ibar}{{\overline{\imath}}}

\newcommand{\Cbar}{{\overline{C}}}

\newcommand{\Gbar}{{\overline{G}}}
\newcommand{\Hbar}{{\overline{H}}}

\newcommand{\Zbar}{{\overline{Z}}}

\makeatletter
\newcommand{\xleftrightarrows}[2][]{\mathrel{
 \raise.40ex\hbox{$
       \ext@arrow 3095\leftarrowfill@{\phantom{#1}}{#2}$}
 \setbox0=\hbox{$\ext@arrow 0359\rightarrowfill@{#1}{\phantom{#2}}$}
 \kern-\wd0 \lower.40ex\box0}}

\newcommand{\xrightleftarrows}[2][]{\mathrel{
 \raise.40ex\hbox{$\ext@arrow 3095\rightarrowfill@{\phantom{#1}}{#2}$}
 \setbox0=\hbox{$\ext@arrow 0359\leftarrowfill@{#1}{\phantom{#2}}$}
 \kern-\wd0 \lower.40ex\box0}}
\newcommand{\xleftrightarrow}[2][]{
     \ext@arrow 0055{\leftrightarrowfill@}{#1}{#2}
}
\def\leftrightarrowfill@{
 \arrowfill@\leftarrow\relbar\rightarrow
}  
\makeatother

\newcommand{\ghilb}{G\operatorname{-Hilb}}

\newcommand{\Ymax}{Y_{\max}}
\newcommand{\cpt}{\text{\rm\small cpt}}
\newcommand{\zdim}{0\text{-dim}}

\newcommand{\ZL}{\operatorname{ZL}}
\newcommand{\Irr}{\operatorname{Irr}}
\newcommand{\Perf}{\operatorname{Perf}}
\newcommand{\rhonat}{\rho_{\text{nat}}}
\title{$G$-constellations and the maximal resolution of a quotient surface singularity}
\author{Akira Ishii}
\date{}
\begin{document}

\maketitle

\begin{abstract}
For a finite subgroup $G$ of $\GL(2, \bC)$, we consider the moduli space $\scM_\theta$ of $G$-constellations.
It depends on the stability parameter $\theta$ and if $\theta$ is generic it is a resolution of singularities of $\bC^2/G$.
In this paper, we show that a resolution $Y$ of $\bC^2/G$ is isomorphic to $\scM_\theta$ for some generic $\theta$ if and only if $Y$ is dominated by the maximal resolution under the assumption that $G$ is abelian or small.

{\em MSC 2010} : 14D20; 	14E16 ; 14J17 
\end{abstract}

%\tableofcontents

\section{Introduction}
 \label{sc:introduction}
The moduli spaces of $G$-constellations (on an affine space) are introduced in \cite{craw-ishii}.
It is a generalization of the Hilbert scheme of $G$-orbits, which is denoted by $\ghilb$.
The moduli space depends on some stability parameter $\theta$ and the moduli space of
$\theta$-stable $G$-constellations is denoted by $\scM_\theta$.
If $G$ is a subgroup of $\SL(n, \bC)$ acting on $\bC^n$ and $n \le 3$,
then $\scM_\theta$ is a crepant resolution of $\bC^n/G$ for a generic stability parameter $\theta$.
The main result of \cite{craw-ishii} is that for a finite abelian subgroup $G \subset \SL(3, \bC)$
and for a projective crepant resolution $Y \to \bC^3/G$, there is a generic stability parameter
$\theta$ such that $Y \cong \scM_\theta$.
See \cite{Oskar}, \cite{Nolla-Sekiya}, \cite{Jung_TQS} and \cite{Jung_CI} for related results.

The purpose of this paper is to consider the case where $G$ 
is a finite subgroup of $\GL(2, \bC)$.
In this case, $\ghilb(\bC^2)$ is the minimal resolution of $\bC^2/G$ by \cite{Ishii_MKG}
but $\scM_\theta$ is a resolution which may not be minimal for generic $\theta$
(as we see in this paper).
Then what is the condition for a resolution $Y \to \bC^2$ to be
isomorphic to some $\scM_\theta$?
One important observation is that there is a fully  faithful functor (see Theorem \ref{thm:fully_faithful})
$$
D^b(\coh \scM_\theta) \hookrightarrow D^b(\coh^G \bC^2)
$$
between the derived categories.
According to the DK hypothesis \cite{Kawamata_DMC}, 
the inclusion of derived categories should be related with
inequalities of canonical divisors.
Then it is natural to ask if the following is true:
$Y$ is isomorphic to $\scM_\theta$ for some $\theta$
if and only if $Y$ is between the minimal and the maximal resolutions
(see Conjecture \ref{conjecture:main}),
 where the maximal resolution means the unique maximal one satisfying the inequality
as in \cite{Kollar-SB}.
The main result of this paper is the following.
Recall that $G$ is said to be small if it contains no pseudo reflection.
\begin{theorem}[=Theorem \ref{thm:small}]\label{thm:main}
Let $G \subset \GL(2, \bC)$ be a finite small subgroup and
let $X=\bC^2/G$ be the quotient singularity.
Then a resolution of singularities $Y \to X$ is isomorphic to $\scM_\theta$ for some $\theta$
if and only if $Y$ is dominated by the maximal resolution.
\end{theorem}
Conjecture \ref{conjecture:main} is a conjecture for general (not necessarily small) finite subgroups where the maximal resolution is defined for the pair of the quotient variety $\bC^2/G$
and the associated boundary divisor.
The ``only if'' part of the conjecture is proved in Proposition \ref{prop:only_if} by using the embedding of $G$ into $\SL(3, \bC)$ and the fact that the moduli space of $G$-constellations for $G\subset \SL(3, \bC)$
is a crepant resolution of $\bC^3/G$.
We can show that the conjecture is true if $G$ is abelian (Theorem \ref{thm:abelian})
by using the result of \cite{craw-ishii}.
The idea in the non-abelian case of Theorem \ref{thm:main} is to use iterated construction
of moduli spaces as in \cite{Ishii-Ito-Nolla} and reduce the problem to the abelian group case.
Namely, let $N$ be the cyclic group generated by $-I$, which is a normal subgroup of every non-abelian finite small subgroup.
We consider $G/N$-constellations on the moduli space of $N$-constellations
in \S \ref{sec:iterated}.
In order to do such iterated constructions, we define $G$-constellations on a general variety
and consider their stability parameters in \S \ref{sec:variety}.
A key to the proof of Theorem \ref{thm:main} is the description of the space of
stability parameters for $G/N$-constellations on the moduli space of $N$-constellations,
which is done in \S \ref{subsec:key}.
The proof of Theorem \ref{thm:main} is completed in \S \ref{subsec:main}.
\subsubsection*{\bf Acknowledgements}
This work depends a lot on the joint work \cite{craw-ishii} and the author thanks
Alastair Craw for stimulating discussions since then.
He is grateful to Seung-Jo Jung and the anonymous referee for many useful comments.
This research was supported in part by JSPS KAKENHI Grant Number 15K04819.

\section{$G$-constellations on $\bC^n$}
%We recall basic definitions and results from \cite{craw-ishii}.
\subsection{Definitions}\label{subsec:def}
Let $V=\bC^n$ be an affine space and $G \subset \GL(V)$ a finite subgroup.
\begin{definition}
A {\it $G$-constellation} on $V$ is a $G$-equivariant coherent sheaf $E$ on $V$
such that $H^0(E)$ is isomorphic to the regular representation of $G$
as a $\bC[G]$-module.
\end{definition}

Let $R(G)=\bigoplus_{\rho \in \Irr(G)} \bZ \rho$ be the representation ring of $G$,
where $\Irr(G)$ denotes the set of irreducible representations of $G$.
The parameter space of stability conditions of $G$-constellations is the $\bQ$-vector space
$$
\Theta=\{\theta \in \Hom_\bZ(R(G), \bQ) \mid \theta(\bC[G])=0\},
$$
where $\bC[G]$ is regarded as the regular representation of $G$.
The definition of the stability is based on the stability of quiver representations \cite{King}:
\begin{definition}
A $G$-constellation $E$ is {\it $\theta$-stable} (or {\it $\theta$-semistable}) if every proper G-equivariant coherent subsheaf 
$0 \subsetneq F \subsetneq E$ satisfies $\theta(H^0(F)) > 0$ (or $\theta(H^0(F)) \ge 0$).
Here the representation space $H^0(F)$ of $G$ is regarded as an element of $R(G)$.
\end{definition}
By virtue of King \cite{King}, there is a fine moduli scheme $\scM_\theta=\scM_\theta(V)$ of $\theta$-stable $G$-constellations on $V$.
\begin{definition}
We say that a parameter $\theta \in \Theta$ is {\it generic} if a $\theta$-semistable $G$-constellation is always $\theta$-stable.
\end{definition}
There is a morphism $\tau:\scM_\theta(V) \to V/G$ which sends a $G$-constellation to its support.
It is a projective morphism if $\theta$ is generic (see \cite[Proposition 2.2]{craw-ishii}).
\subsection{Results of \cite{craw-ishii}}
In this subsection,
we recall results from \cite{craw-ishii}.
Suppose $V= \bC^3$ and let $G \subset \SL(V)$
be a finite abelian subgroup.
For a generic parameter $\theta \in \Theta$,
the morphism
$$
\tau: \scM_\theta \to \bC^3/G
$$
is a projective crepant resolution and
we have a Fourier-Mukai transform
$$
\Phi_\theta: D^b(\coh \scM_\theta) \overset{\sim}{\longrightarrow} D^b(\coh^G(\bC^3)).
$$
Here for a variety $Y$, $\coh Y$ denotes the category of coherent sheaves on $Y$
and if $Y$ is acted on by a finite group $G$, $\coh^G(Y)$ denotes the category
of $G$-equivariant coherent sheaves on $Y$.
The subset of $\Theta$ consisting of generic parameters is divided into chambers;
the moduli space $\scM_\theta$ and the equivalence $\Phi_\theta$ depend only on the chamber
to which $\theta$ belongs.
Thus we write $\scM_C$ and $\Phi_C$ instead of $\scM_\theta$ and $\Phi_\theta$
where $C$ is the chamber that contains $\theta$.
We write
$$\varphi_C: K(\coh_{0} \scM_C) \to K(\coh^G_0(\bC^3))$$
for the induced isomorphism of the Grothendieck groups of the full subcategories
$\coh_{0} \scM_\theta$ and $\coh^G_0(\bC^3)$
consisting of sheaves supported on the subsets $\tau^{-1}(0)$ and on $\{0\}$ respectively.
Since $K(\coh^G_0(\bC^3))$ has a basis consisting of skyscraper sheaves $\scO_0 \otimes \rho$ with $\rho \in \Irr(G)$, it is naturally identified with $R(G)$.

The dual of $\varphi_C$ is regarded as the map
$$
\varphi_C^*: K(\coh^G(\bC^3)) \to K(\coh \scM_\theta)
$$
between the Grothendieck groups of the categories of sheaves without restrictions on the supports.
Then $K(\coh^G(\bC^3))$ is identified with $\Hom(R(G), \bZ)$
and $\varphi_C^*$ induces an isomorphism
$$
\Theta \overset{\sim}{\longrightarrow} F^1K(\coh \scM_\theta)_\bQ,
$$
where $F^iK(\coh \scM_\theta)$ is the subgroup consisting of the classes of objects whose supports are at least of codimension $i$.

On $\scM_C$ there are tautological bundles $\scR_\rho$ for irreducible representations
$\rho$ such that $\bigoplus_\rho \scR_\rho \otimes_{\bC} \rho$ has a structure
of the universal $G$-constellation.
For $\theta \in C$,
$$
\scL_C(\theta):= \bigotimes_\rho (\det \scR_\rho)^{\otimes \theta(\rho)}
$$
is the (fractional) ample line bundle on $\scM_\theta$ obtained by the GIT construction.
It coincides with the class
\begin{equation}\label{eq:FM(theta)}
[\varphi_C^*(\theta)] \in F^1K(\coh \scM_C)_\bQ/F^2K(\coh \scM_C)_\bQ \cong \Pic(\scM_C)_\bQ
\end{equation}
as in \cite[\S 5.1]{craw-ishii}.
Hence $[\varphi_C^*(\theta)] \in \Amp(\scM_C)$ where $\Amp(\scM_C)$ is the ample cone considered in $\Pic(\scM_C)_\bQ$.
The main theorem of \cite{craw-ishii} and the argument in \cite[\S 8]{craw-ishii} shows the following:
\begin{theorem}[\cite{craw-ishii}]\label{thm:craw-ishii}
For any projective crepant resolution $Y \to \bC^3/G$ and a class $l \in \Amp(Y)$, there exist a chamber $C$ with $Y \cong \scM_C$ and a parameter $\theta \in C$ satisfying $l=[\varphi_C^*(\theta)]$.
\end{theorem}
\begin{proof}
The existence of a chamber $C$ such that $Y \cong \scM_C$ is \cite[Theorem 1.1]{craw-ishii}.
Moreover, \cite[Proposition 8.2]{craw-ishii} ensures that
we can find a chamber $C$ and a parameter $\theta \in \Cbar$ with $l=[\varphi_C^*(\theta)]$.
Suppose $\theta \in \Cbar \setminus C$.
We have to see we can perturb $\theta$ in the fiber of
$p \circ \varphi_C^*$ so that $\theta$ is in some chamber, where
$$
p: F^1K(\coh \scM_C)_\bQ \to \Pic(\scM_C)_\bQ
$$
is the projection.
Here recall that a wall of the chamber $C$ is either the preimage of a wall of the ample cone
by $p \circ \varphi_C^*$ (type I or III) or does not contain a fiber of $p \circ \varphi_C^*$
(type $0$); see \cite[Theorem 5.9]{craw-ishii}.
In our case, $p\circ \phi_C^*(\theta)=l$ is ample and therefore $\theta$ is on walls of type $0$.
Since the images of adjacent chambers in $F^1K(\coh \scM_C)_\bQ$ are related as in
\cite[(8,2) or (8.3)]{craw-ishii},
we can perturb $\theta$ in the fiber of $p \circ \varphi_C^*$
and go out of walls.
\end{proof}

\subsection{$G$-constellations on $\bC^2$}
Let $G$ be a finite subgroup of $\GL(2, \bC)$.
\begin{theorem}\label{thm:fully_faithful}
If $\theta$ is generic, then the moduli space $\scM_\theta$ is a resolution of singularities of $\bC^2/G$.
Moreover, the universal family of $G$-constellations defines a fully faithful functor
$$
\Phi_\theta: D^b(\coh \scM_\theta) \to D^b(\coh^G \bC^2).
$$
\end{theorem}
\begin{proof}
This is essentially Theorem 1.3 in the first arXiv version of \cite{Bridgeland-King-Reid}.
We have the inequality
$$
\dim \scM_\theta \times_{(\bC^2/G)} \scM_\theta \le \dim \bC^2
$$
which is sharper than the assumption in \cite{Bridgeland-King-Reid}.
This allows us to apply the argument of \cite{Bridgeland-King-Reid}
(without using the triviality of the Serre functors)
to show that $\Phi_\theta$ is fully faithful and that
$\scM_\theta$ is smooth and connected
(see \cite[Theorem 6.2]{Ishii_MKG}).
\end{proof}
The problem we consider is to characterize the resolutions $Y$ such that $Y \cong \scM_\theta$ for some generic $\theta$.

\section{The maximal resolution}
Let $G$ be a finite subgroup of $\GL(2, \bC)$,
which is not necessarily small,
i.e., the action may not be free on $\bC^2 \setminus \{0\}$.
Then the quotient variety $X=\bC^2/G$ is equipped with a boundary divisor $B$
determined by the equality $\pi^*(K_X+B)=K_{\bC^2}$.
More precisely, $B$ is expressed as
$$
B = \sum_j \frac{m_j-1}{m_j} B_j,
$$
where $B_j \subset X$ is the image of a one-dimensional linear subspace
whose point-wise stabilizer subgroup $G_j \subset G$ is cyclic of order $m_j$.
Note that $G$ is small
if and only if $B=0$.
Let $\tau: Y \to X$ be a resolution of singularities and write
\begin{equation}\label{eq:discrepancy}
K_Y+\tau_*^{-1}B\equiv \tau^*(K_X+B) + \sum_i a_i E_i,
\end{equation}
where $E_i$ are the exceptional divisors and $a_i \in \bQ$.
Recall that $(X, B)$ is a KLT pair (\cite[Proposition 5.20]{Kollar-Mori}), which implies $a_i > -1$ for all $i$.
Then among the resolutions $Y$ which satisfies $a_i \le 0$ for all $i$, there is a unique maximal one, as in \cite{Kollar-SB} (see also \cite[Theorem 17]{Kawamata_DMC}).
It is called the maximal resolution of $(X, B)$ and we denote it by $\Ymax$.

Notice that the system of inequalities $a_i \le 0$ is an inequality between canonical divisors.
According to the DK-hypothesis \cite{Kawamata_DMC},
the inequality should correspond to the embedding of derived categories in Theorem \ref{thm:fully_faithful} with $Y=\scM_\theta$.
Thus we make the following conjecture:
 
\begin{conjecture}\label{conjecture:main}
Let $G \subset \GL(2, \bC)$ be a finite subgroup and
consider the quotient $X=\bC^2/G$ with the boundary divisor $B$.
For any resolution of singularities $Y \to X$,
there is a generic $\theta \in \Theta$ with $Y \cong \scM_\theta$
if and only if there is a morphism $\Ymax \to Y$ over $X$.
Here $\Ymax$ is the maximal resolution of $(X,B)$.
\end{conjecture}

\section{``Only if'' part}
In this section, we show the ``only if'' part of Conjecture \ref{conjecture:main}.
Embed $\GL(2, \bC)$ into $\SL(3, \bC)$ by
sending a matrix $A \in \GL(2, \bC)$ to
$\begin{pmatrix} A & 0 \\ 0 & \det(A)^{-1} \end{pmatrix}$.
Then for $\theta \in \Theta$, we can consider the moduli space $\scM_\theta(\bC^3)$ of $\theta$-stable $G$-constellations on $\bC^3$ with respect to the action of $G$ on $\bC^3$.
\begin{lemma}\label{lem:embed_to_M_theta(C^3)}
For any $\theta \in \Theta$, there is a closed embedding $\scM_\theta \hookrightarrow \scM_\theta(\bC^3)$ which fits into the commutative diagram
\[
    \begin{tikzcd}
    \scM_\theta \arrow[swap]{d}{  }\arrow[r, hookrightarrow]{  } & \scM_\theta(\bC^3) \arrow{d}{  }\\
    \bC^2/G \arrow[r, hookrightarrow]{  } & \bC^3/G.
    \end{tikzcd}
\]
Moreover, if $\theta$ is generic for $G$-constellations on $\bC^3$, then
the vertical arrows are projective and hence are resolutions of singularities.
\end{lemma}
\begin{proof}
%Let $j: \bC^2 \hookrightarrow \bC^3$ be the embedding.
%Then the fully faithful functor $j_*: \coh^G(\bC^2) \to \coh^G(\bC^3)$
%sends $G$-constellations to $G$-constellations and preserves $\theta$-stability.
Recall that the universal family of $G$-constellations on $\bC^3$ is given by
the tautological bundles $\{\scR_\rho\}_{\rho \in \Irr G}$ and the $G$-equivariant morphism
\begin{align}\label{eq:universal_rep}
\bigoplus_\rho \scR_\rho \otimes_{\bC} \rho \longrightarrow
\bC^3 \otimes \left(\bigoplus_\rho \scR_\rho \otimes_{\bC} \rho\right).
\end{align}
If $\rhonat$ denotes the representation given by $G \subset \GL(2, \bC)$,
then $\bC^3$ above is $\rhonat \oplus \det \rhonat^*$.
Taking the third coordinate of $\bC^3$ in \eqref{eq:universal_rep}
we obtain a morphism
\[
z_\rho: \scR_\rho \to \scR_{\rho \otimes \det\rhonat}
\]
for each $\rho$.
It is straightforward that the scheme theoretic intersection of the zero loci of $z_\rho$'s
is isomorphic to $\scM_\theta$.
Hence $\scM_\theta$ is a closed subscheme of $\scM_\theta(\bC^3)$.
Moreover, we can see that the composite $\scM_\theta \hookrightarrow \scM_\theta(\bC^3) \to \bC^3/G$ factors through $\bC^2/G$.
If $\theta$ is generic for $G$-constellations on $\bC^3$, then it is also generic for $G$-constellations on $\bC^2$, from which the projectivities of the vertical arrows follow.
\end{proof}
Now let us prove the ``only if'' part.
\begin{proposition}\label{prop:only_if}
If $\theta$ is generic, then there is a morphism $\Ymax \to \scM_\theta$ over $X$.
\end{proposition}
\begin{proof}
Putting $Y=\scM_\theta$,
we show that $a_i \le 0$ for all $i$ in \eqref{eq:discrepancy}.
Embed $G$ into $\SL(3, \bC)$ and consider $U:=\scM_{\theta}(\bC^3)$,
the moduli space of $\theta$-stable $G$-constellations on $\bC^3$.
Here, we may assume that $\theta$ is generic for $G$-constellations on $\bC^3$
by slightly perturbing $\theta$ if necessary.
Then $U$ is a crepant resolution of $\bC^3/G$ containing $Y$ by Lemma \ref{lem:embed_to_M_theta(C^3)} and therefore we have
\begin{equation}\label{eq:K_Y}
K_Y\cong \scO_U(Y)|_Y.
\end{equation}
Let $z$ be the coordinate function of $\bC^3$ such that $\bC^2\subset \bC^3$ is defined by $z=0$.
Then $z^n$ is invariant under the action of $G$ where $n$ is the order of $G$.
We claim that the principal divisor $(z^n)$ on $U$ is of the form
\begin{equation}\label{eq:z^n}
(z^n)=nY + \sum_j \frac{n(m_j-1)}{m_j} B'_j + \sum_k d_k D_k
\end{equation}
where $B'_j, D_k\subset U$ are prime divisors such that
$B'_j \cap Y = \tau_*^{-1} B_j$ and $D_k \cap Y$ is contained in the exceptional locus
of $Y \to \bC^2/G$ (or empty).
This is saying that there exists an exceptional prime divisor
$B'_j$ of $U\to \bC^3/G$ lying over $B_j$ with $B'_j \cap Y = \tau_*^{-1} B_j$
and that its coefficient in $(z^n)$ is $\frac{n(m_j-1)}{m_j}$.
We may check this over the complete local ring
$\widehat{\cO}_{\bC^3/G,P}$ at a point $P \in B_j \setminus \{0\}$. 
Since $G_j$ is the stabilizer subgroup of a point of $\bC^3$ lying over $P$, there is an isomorphism of complete local rings:
\[
\widehat{\cO}_{\bC^3/G, P} \cong \widehat{\cO}_{\bC^3/G_j, [0]}.
\]
Let $\Btilde_j$ be a line in $\bC^2$ mapped to $B_j$
and take a $G_j$-invariant linear subspace $\Btilde_j^\perp$
of $\bC^3$ such that
\[
\bC^3=\Btilde_j \times \Btilde_j^\perp.
\]
Then $G_j \cong \bZ/m_j \bZ$ is a subgroup of $\{1\}\times \SL(\Btilde_j^\perp)$
and therefore we have
\[
\bC^3/G_j\cong \Btilde_j \times (\Btilde_j^{\perp} /G_j),
\]
where $\Btilde_j^\perp/G_j$ is a rational double point of type $A_{m_j-1}$.
Thus we can see that on the crepant resolution
\[
U \times_{(\bC^3/G)} \Spec \widehat{\cO}_{\bC^3/G, P}
\to \Spec \widehat{\cO}_{\bC^3/G, P} \cong \Spec \widehat{\cO}_{\bC^3/G_j, [0]},
\]
there is a prime divisor $\widehat B_j'$ with desired properties
such that
the coefficient of $\widehat B_j'$ in the divisor $(z^{m_j})$ is $m_j-1$.
Since $m_j$ divides $n$, this proves \eqref{eq:z^n}.
%%%%%%%%%%
%%%%%%%%%%
%Let $\Btilde_j$ be a line in $\bC^2$ mapped to $B_j$.
%Then the subgroup
%$$
%K:=\{g \in G \mid g(\Btilde_j)=\Btilde_j \}
%$$
%is abelian since it can be diagonalized.
%Moreover, an analytic neighborhood of $B_j \setminus \{0\}$ in $\bC^3 \setminus \{0\}$
%and its preimage in $U$ are
%isomorphic to the quotient of an analytic neighborhood of $\Btilde_j \setminus \{0\}$
%by the $K$-action and its crepant resolution respectively.
%Therefore \eqref{eq:z^n} can be reduced to the abelian group case,
%which can be easily checked by the toric method.

From \eqref{eq:K_Y} and \eqref{eq:z^n}, we obtain
$$
K_Y+\tau_*^{-1}B \equiv -\sum \frac{d_k}{n}(D_k \cap Y).
$$
Here, note that $z^n$ is a regular function and therefore the coefficients 
in \eqref{eq:z^n} are all non-negative.
Especially, we have $d_k \ge 0$ for all $k$.
This proves the assertion since $K_X+B\in\Pic(X) \otimes \bQ=0$
in \eqref{eq:discrepancy}.
\end{proof}

\section{Abelian group case}
Let $G\subset \GL(2, \bC)$ be a finite abelian subgroup of order $n$.
As in the previous section, we  embed $G\subset \GL(2, \bC)$ into $\SL(3, \bC)$.

\begin{theorem}\label{thm:abelian}
Conjecture \ref{conjecture:main} is true if $G$ is abelian.
\end{theorem}
\begin{proof}
It is sufficient to prove the ``if'' part by Proposition \ref{prop:only_if}.
Let $Y \to X=\bC^2/G$ be a resolution which is dominated by $\Ymax$.
By Proposition \ref{prop:ample} below, there is a projective crepant resolution $U \to \bC^3/G$ such that $Y \subset U$.
Then \cite{craw-ishii} ensures that there is a generic parameter $\theta$ such that
$U \cong \scM_{\theta}(\bC^3)$.
Then $\scM_\theta(\bC^2)$ is isomorphic to $Y$ by Lemma \ref{lem:embed_to_M_theta(C^3)}.
\end{proof}

Before stating the proposition, we need some notation.
We diagonalize $G$ and write
$$
g=\diag (\zeta_n^{a_g}, \zeta_n^{b_g})
$$
for $g \in G$
where $\zeta_n$ is a primitive $n$-th root of unity.
Put
\begin{align*}
N_2:&= \bZ^2 + \sum_{g \in G} \bZ \cdot \frac{1}{n}(a_g, b_g),\\
N_3:&= \bZ^3 + \sum_{g \in G} \bZ \cdot \frac{1}{n}(a_g, b_g, -a_g-b_g)
\end{align*}
which are the lattices of one-parameter subgroups for the toric varieties
$\bC^2/G$ and $\bC^3/G$ respectively.
The {\it junior simplex} $\Delta \subset (N_3)_\bR$ is the triangle with vertices $e_1, e_2, e_3$ where $\{e_1, e_2, e_3\}$ is the basis of $\bZ^3$ with $e_1, e_2 \in \bZ^2$.
A crepant resolution $U$ corresponds to a basic triangulation of $\Delta$.
For  a basic triangulation $\Sigma$ of $\Delta$,
let $U_\Sigma$ be the corresponding crepant resolution.

Consider the natural projection
$$
p_{12}:N_3 \to N_2
$$
and put $\Delta':=p_{12}(\Delta) \cong \Delta$.
Let $e_i' \in (\bR_{\ge 0})e_i \cap N_2$ be the primitive vector
and write $e_i=m_i e_i'$ for $i=1,2$.
If $B_i \subset \bC^2/G$ denote the divisor corresponding to $e_i'$,
then
$$B:=\frac{m_1-1}{m_1}B_1 + \frac{m_2-1}{m_2}B_2$$
is the boundary divisor for the quotient $\bC^2/G$.
A resolution $Y$ of $\bC^2/G$ % dominated by the maximal resolution $\Ymax$ of $(\bC^2/G, B)$
is given by choosing primitive vectors
$v_0, v_1, \dots, v_s$ of $(\bZ_{\ge 0})^2 \cap N_2$
such that $v_0 =e_1'$, $v_s=e_2'$ and $\{v_{i-1}, v_i\}$ is a basis of $N_2$
for $i=1,\dots, s$.
If $E_i$ denotes the exceptional divisor corresponding to $v_i$ for $i=1, \dots, s-1$,
then the discrepancy $a_i$ of $E_i$ for the pair $(X, B)$ is $\alpha_i + \beta_i-1$ where $v_i=(\alpha_i, \beta_i)$.
Therefore, $Y$ is dominated by the maximal resolution $\Ymax$ of $(X, B)$ if and only if all of $v_1, \dots, v_{s-1}$ are in $\Delta'$.

Let $G_{(1,0)} \subset G$ be the stabilizer subgroup of $(1,0)  \in \bC^2 = \bC^2 \times \{0\} \subset \bC^3$.
Then $G_{(1,0)}$ acts on $\{1\}\times \bC^2 \cong \bC^2$ as a subgroup of $\SL(2)$
and the quotient $(\{1\}\times\bC^2)/G_{(1,0)}$ is a closed subvariety of $\bC^3/G$.
Let
$$W \to (\{1\}\times\bC^2)/G_{(1,0)}$$
be the minimal resolution.
Notice that $W$ is contained
in any crepant resolution $U$ of $\bC^3/G$
since  $(\{1\}\times\bC^2)/G_{(1,0)} \subset \bC^3/G$
is transversal to the one-dimensional stratum $(\bC^\times \times \{(0, 0)\})/G$.
Now we prove the following proposition.
The surjectivity of the ample cones will be used in the proof of the main theorem.

\begin{proposition}\label{prop:ample}
Let $Y \to \bC^2/G$ be a resolution dominated by $\Ymax$.
Then there is a projective crepant resolution
$U=U_\Sigma \to \bC^3/G$ containing $Y$
such that the restriction map $\Amp(U) \to \Amp(W)$  of the ample cones 
is surjective.
\end{proposition}
\begin{proof}
Since $Y$ is dominated by $\Ymax$, it is defined by
primitive vectors $v_0, v_1, \dots, v_s \in \Delta' \cap N_2$.
Let $w_i \in \Delta \cap N_3$ be the unique lift of $v_i$.
For  a basic triangulation $\Sigma$ of $\Delta$,
$U=U_\Sigma$ contains $Y$ if and only if the points connected to $e_3$ in $\Sigma$
are exactly $w_0, \dots, w_s$.

We prove the assertion by the induction on the order $|G|$ of $G$.
If $|G|=1$, then there is nothing to prove.
We consider the number
$$\nu:=\#(\{w_0, \dots, w_{s-1}\}\setminus \{e_1\})\ge 0.$$

If $\nu=0$, then $s$ must be $1$ and $w_0=e_1$ is a primitive vector.
Especially, $\{e_1, v_1\}$ is a basis of $N_2$.
In this case,  $\Delta$ has a unique basic triangulation $\Sigma$ and
$U_\Sigma \cong W \times \bC$.
Hence the restriction map
$\Amp(U_\Sigma) \to \Amp(W)$ is an isomorphism.

Suppose $\nu >0$.
Let $w \in \{w_0, \dots, w_{s-1}\}\setminus \{e_1\}$
be a point such that
the coefficient of $e_3$ in $w$ is the smallest.
Then $w$ determines a {\it star subdivision} of $\Delta$:
$\Delta=\bigcup_{i=1}^{3} \Delta_i$
where $\Delta_1$, $\Delta_2$, $\Delta_3$
are the triangles $we_2e_3$, $we_1e_3$, $we_1e_2$ respectively.
Note that either $\Delta_2$ or $\Delta_3$ may be degenerate,
in which case we simply ignore the degenerate one in the sequel.
This subdivision of $\Delta$, which is denoted by $\Sigma_0$, determines a projective crepant birational morphism
$U_{\Sigma_0} \to \bC^3/G$ where $U_{\Sigma_0}$ is a toric variety with at most Gorenstein quotient singularities.
The choice of $w$ implies that $w_0, \dots, w_s$ are in $\Delta_1 \cup \Delta_2$.
Hence by the induction hypothesis, there are basic triangulations $\Sigma_1$ and $\Sigma_2$ of $\Delta_1$ and $\Delta_2$ respectively,
which satisfy the following conditions:
in $\Sigma_1 \cup \Sigma_2$,
the vertices connected to $e_3$ are exactly $w_0, \dots, w_s$,
the map $\Amp(U_{\Sigma_1}) \to \Amp(W)$ is surjective
and $\Amp(U_{\Sigma_2})$ is non-empty.
We choose an arbitrary basic triangulation $\Sigma_3$ of $\Delta_3$ with non-empty
$\Amp(U_{\Sigma_3})$.
Combining the triangulations $\Sigma_1$, $\Sigma_2$ and $\Sigma_3$ together,
we obtain a basic triangulation of $\Delta$ such that $U_\Sigma \supset Y$.
Since $\Delta=\bigcup_{i=1}^{3} \Delta_i$ is a star subdivision,
we see that $U_\Sigma \to U_{\Sigma_0}$
is a projective morphism and
the map $\Amp(U_\Sigma) \to \Amp(U_{\Sigma_1})$ is surjective.
Therefore, the morphism $U_\Sigma \to \bC^3/G$ is also projective
and $\Amp(U_\Sigma) \to \Amp(W)$ is surjective.
\end{proof}
%%%%%%%%%%%%

\section{$G$-constellations on a variety}\label{sec:variety}
In the case of $G$-constellations for non-abelian $G\subset \GL(2, \bC)$,
we shall use the iterated construction of moduli spaces for a normal subgroup of $G$
as in \cite{Ishii-Ito-Nolla}.
In order to do so, we have to consider $G$-constellations on a variety,
rather than an affine space.
Especially, the space of stability parameters will be larger than the affine case in general.

Suppose $U$ is a quasi projective variety of finite type over $\bC$
and $G$ is a finite group acting on $U$.
Let $\coh^G(U)$ be the abelian category of $G$-equivariant coherent sheaves on $U$
and $\coh^G_\cpt(U)$ its subcategory consisting of sheaves whose supports are proper over $\bC$.
The corresponding Grothendieck groups are denoted by $K(\coh^G(U))$ and $K(\coh^G_\cpt(U))$ respectively.
We also consider the perfect derived category $\Perf^G(U)$ of $G$-equivariant perfect complexes
and its Grothendieck group $K(\Perf^G(U))$.
For $\alpha \in K(\Perf^G(U))$ and $\beta \in K(\coh^G_\cpt(U))$, we write
\begin{equation}\label{eq:pairing}
\chi(\alpha, \beta):=\sum_i(-1)^i \dim \Ext^i_{\scO_U}(\alpha, \beta)^G.
\end{equation}

Let $\coh^G_{\zdim}(U)$ be the subcategory of $\coh^G_\cpt(U)$ consisting of sheaves with $0$-dimensional support.
We define the stability condition of objects in $\coh^G_{\zdim}(U)$.
\begin{definition}\label{def:stability}
Fix a class $\xi \in K(\Perf^G(U))$.
 An object $E\in \coh^G_{\zdim}(U)$ is said to be {\it $\xi$-stable} (or {\it $\xi$-semistable}) if $\chi(\xi, E)=0$ and if for every non-trivial $G$-equivariant subsheaf $F$ of $E$, $\chi(\xi, [F])>0$ (or $\chi(\xi, [F])\ge 0$).
\end{definition}

In the case where $U=\bC^N$ is an affine space with a linear $G$-action,
$K(\Perf^G(U))=K(\coh^G(U))$ is isomorphic to (the dual of) the representation ring $R(G)$
and the definition coincides with the ($\bZ$-valued) one in \S \ref{subsec:def}.

We have a well-defined function $\rank: K(\Perf^G(U)) \to \bZ$
which extends the rank of a locally free sheaf.
Put
$$
K(\Perf^G(U))^0:= \{ \xi \in K(\Perf^G(U)) \mid \rank \xi=0\}.
$$

\begin{definition}\label{def:G-constellation}
A {\it $G$-constellation} on $U$ is a $G$-equivariant coherent sheaf $E$ on $U$ with finite support
such that $H^0(E)$ is isomorphic to the regular representation of $G$
as a representation of $G$
and $\chi(\xi, E)=0$ for any $\xi \in K(\Perf^G(U))^0$.
 \end{definition}

For any $\xi \in K(\Perf^G(U))^0$, we can discuss the $\xi$-(semi)stabilities of $G$-constellations on $U$ according to Definition \ref{def:stability}.
Since the multiplication by a positive integer does not change the stability condition,
we may replace $K(\Perf^G(U))^0$ by $K(\Perf^G(U))^0_\bQ$.

\begin{remark}
In general, there may exist an object
$E$ supported on several fixed points such that $H^0(E) \cong R(G)$ but $\chi(\xi, E) \ne 0$
for some $\xi \in K(\Perf^G(U))^0$.
Definition \ref{def:G-constellation} excludes such cases.
\end{remark}

\begin{remark}
If $U$ is smooth, then $K(\Perf^G(U))$ coincides with $K(\coh^G(U))$ and
we write $K(\coh^G(U))^0$ instead of $K(\Perf^G(U))^0$.
\end{remark}

Now we define the moduli functors of $G$-constellations:
\begin{definition}
Fix a class $\xi \in K(\Perf^G(U))^0_\bQ$.
Then the moduli functor for the $\xi$-stable $G$-constellations on $U$ is defined to be the functor
\[
S \mapsto
\{ \text{flat families of $\xi$-stable $G$-constellations parameterized by $S$}\}/\sim
\]
for a locally noetherian scheme $S$ over $\bC$ where $E_S \sim F_S$ for flat families $E_S$ and $F_S$ means that there is a line bundle $L$ on $S$ such that
$E_S \cong F_S \otimes L$.
\end{definition}
\begin{remark}
We show the existence of the moduli scheme in a very special case in Theorem \ref{thm:iterated_construction}.
We do not discuss the existence problem in a general case in this paper.
\end{remark}

\section{Iterated construction of moduli spaces}\label{sec:iterated}
In this section, let $V$ denote either $\bC^2$ or $\bC^3$
and consider a finite subgroup $G \subset \GL(V)$ with a normal subgroup $N$ of $G$
such that $N \subset \SL(V)$.
Let
$$\theta^N :R(N) \to \bZ$$
be a generic stability parameter for $N$-constellations on $V$,
which is fixed by the conjugate action of $G$ on $R(N)$.
Put $Y_N=\scM_{\theta^N}(V)$ and $\Gbar=G/N$.
Since $N \subset \SL(V)$ and $\dim V \le 3$, there is an equivalence
\begin{equation}\label{eq:equivalence}
\Phi: D^b(\coh^{\Gbar}(Y_N)) \cong D^b(\coh^G(V))
\end{equation}
as in \cite[Theorem 4.1]{Ishii-Ueda_SMEC}
defined by
$$
\Phi(-)=\bR(p_{V})_*((p_{Y_N})^*(-) \otimes \scU)
$$
where $p_V, p_{Y_N}$ are the projections of $Y_N \times  V$ and $\scU$ is the universal family
of $N$-constellations.

\begin{lemma}\label{lem:Gbar-const}
Let $\scE$ be a $\Gbar$-equivariant coherent sheaf  on $Y_N$ with finite support.
Then $\scE$ is a $\Gbar$-constellation on $Y_N$ if and only if $\Phi(\scE)$ is a $G$-constellation on $V$.
In this case, $\Phi(\scE)$ is $\theta^N$-semistable.
\end{lemma}
\begin{proof}
By the definition of $\Phi$, we can see that $\Phi(\scE)$ is a $0$-dimensional sheaf.
Since $\Phi$ is an equivalence, we have $\chi(\xi, \scE)=\chi(\Phi(\xi), \Phi(\scE))$.
Moreover, we can see $\rank \xi=\rank \Phi(\xi)$ for any $\xi \in K(\coh^{\Gbar}(Y_N))$.
Therefore, if $\scE$ is a $\Gbar$-constellation, $\chi(\xi, \Phi(\scE))=0$ for any
$\xi\in K(\coh^G(V))^0$.
This implies that $H^0(\Phi(\scE))$ is a multiple of the regular representation $\bC[G]$.
If we regard $\scE$ as an object of $\coh(Y_N)$, it is an Artinian sheaf of length $|\Gbar|$
and therefore $\Phi(\scE)$ as an object of $\coh^N(V)$ has a filtration of length $|\Gbar|$
whose factors are $\theta^N$-stable $N$-constellations.
Therefore, $\Phi(\scE)$ is $\theta^N$-semistable and $H^0(\Phi(\scE))$ as a representation of $N$ is the direct sum of $|\Gbar|$ copies of the regular representation of $N$.
This implies that $H^0(\Phi(\scE)) \cong \bC[G]$ and therefore $\Phi(\scE)$ is a $G$-constellation.
The converse is proved in the same way.
\end{proof}

The following lemma follows from the arguments in \cite[\S 8]{Bridgeland-King-Reid}:
\begin{lemma}\label{lem:BKR}
Let $E$ be an $N$-equivariant coherent sheaf on $V$ with finite support such that
$H^0(E)$ is isomorphic to $\bC[N]^{\oplus s}$ for some integer $s>0$ as a $\bC[N]$-module.
If $E$ is $\theta^N$-stable,
then we have $s=1$, i.e., $E$ is an $N$-constellation.
\end{lemma}

We compose $\theta^N$ with the restriction map $R(G) \to R(N)$ and regard it as
a stability parameter for $G$-constellations as in \cite[\S 2.2]{Ishii-Ito-Nolla}.

\begin{lemma}\label{lem:HN}
Let $E$ be a $G$-equivariant coherent sheaf on $V$ with finite support such that
$H^0(E) \cong \bZ[G]^{\oplus s}$ for some $s$.
If $E$ is $\theta^N$-semistable in $\coh^G(V)$, then it is also $\theta^N$-semistable in $\coh^N(V)$.
\end{lemma}
\begin{proof}
Let $\eta: R(N) \to \bZ$ be a group homomorphism such that $\eta(\rho)>0$ for any irreducible representation $\rho$ of $N$.
We further suppose $\eta$ is invariant under the conjugate action of $G$.
Then,
$$
Z(E):=\theta^N(H^0(E))+\sqrt{-1}\eta(H^0(E))
$$
defines a $G$-invariant Bridgeland stability condition \cite[Example 5.5]{Bridgeland_SCTC} (see also \cite[Lemma 7.1.3]{Bayer_Craw_Zhang}) on $\coh^N(V)_0$,
the category of $N$-equivariant coherent sheaves on $V$ with $0$-dimensional support.
As in \cite[Lemma 7.1.5]{Bayer_Craw_Zhang},
the equality $\theta^N(H^0(E))=0$ implies that $E$ is $\theta^N$-semistable if and only if it is semistable with respect to $Z$.
Assume $E$ is not $\theta^N$-semistable and let $F\subset E$ be the first step of
the Harder-Narasimhan filtration of $E$ in $\coh^N(E)$ with respect to $Z$.
Then the uniqueness of the HN filtration and the $G$-invariance of $Z$ imply that
$F$ is invariant under the $G$-action.
This means that $F$ is a subsheaf of $E$ in $\coh^G(V)$,
which contradicts the $\theta^N$-semistability of $E$ in $\coh^G(V)$.
\end{proof}

\begin{proposition}\label{prop:bijection}
The functor $\Phi$ induces a bijection
from the set of $\Gbar$-constellations on $Y_N$
to the set of $\theta^N$-semistable $G$-constellations on $V$.
\end{proposition}
\begin{proof}
If $\scE$ is a $\Gbar$-constellation on $Y_N$, then $\Phi(\scE)$ is a $\theta^N$-semistable $G$-constellation by Lemma \ref{lem:Gbar-const}.
Conversely, suppose $E$ is a $\theta^N$-semistable $G$-constellation on $V$.
By Lemma \ref{lem:Gbar-const}, it suffices to show that $\Phi^{-1}(E)$ lies in $\coh^{\Gbar}(Y_N)$ and has a $0$-dimensional support.
For this purpose, we may regard $\Phi$ as an equivalence $D^b(\coh Y_N) \cong D^b(\coh^N(V))$.
By lemma \ref{lem:HN}, $E$ is $\theta^N$-semistable as a sheaf in $\coh^N(V)$
and therefore has a filtration whose factors are $\theta^N$-stable $N$-constellations by Lemma \ref{lem:BKR}.
Then, $\Phi^{-1}(E)$ as an object in $D^b(\coh(Y_N))$ is a sheaf with a filtration
whose factors are skyscraper sheaves.
This is what we needed.
\end{proof}

Let
$$
\varphi: K(\coh^{\Gbar}(Y_N))^0_\bQ \overset{\sim}{\rightarrow}
K(\coh^G(V))^0_\bQ \cong \Theta
$$
be the isomorphism induced by $\Phi$.
The following theorem  generalizes \cite[Theorem 2.6]{Ishii-Ito-Nolla}.
%%%%% Theorem
\begin{theorem}\label{thm:iterated_construction}
Let $\theta^N: R(N) \to \bZ$ be a generic stability condition for $N$-constellations
fixed by the conjugate action of $G$
and $\xi \in K(\coh^{\Gbar}(Y_N))^0$ be a stability parameter
for $\Gbar$-constellations on $Y_N$.
%We regard $\xi$ as a map $K(\coh_\cpt(Y_N)) \to \bZ$ by the paring $\chi(-, -)$.
\begin{enumerate}
\item[(1)] There exists a scheme $\cM_\xi(Y_N)$ representing the moduli functor for $\xi$-stable $\Gbar$-constellations on $Y_N$.
\item[(2)]
If we put
$$
\theta:=m\theta^N + \varphi(\xi)
$$
for $m \gg 0$, then $\scM_\theta(V)$ is isomorphic to the moduli space
$\scM_\xi(Y_N)$ of $\xi$-stable $\Gbar$-constellations on $Y_N$.
\end{enumerate}
\end{theorem}
%%%%%% end theorem
\begin{proof}
What we prove is that $\scM_\theta(V)$ in (2) represents the moduli functor in (1). 
We choose $m$ so that
$$
m > \sum_{\rho \in \Irr(G)} |(\varphi(\xi))(\rho)|\dim \rho.
$$
Then for any subsheaf $F$ of a $G$-constellation, we have $|(\varphi(\xi))(F)| < m$.

Let $\scE$ be a $\xi$-stable $\Gbar$-constellation on $Y_N$.
Then $\Phi(\scE)$ is a $\theta^N$-semistable $G$-constellation by Proposition \ref{prop:bijection}.
Therefore, a subsheaf $F$ of
$\Phi(\scE)$ satisfies $\theta^N(F) \ge 0$.
If $\theta^N(F) > 0$, then we have $\theta(F)>0$
by our choice of $m$.
If $\theta^N(F)= 0$, then there is a subsheaf $\scF$ of $\scE$ such that $F=\Phi(\scF)$ as in \cite[Lemma 2.6]{Ishii-Ito-Nolla}.
Then we obtain $\theta(F)= \chi(\xi, \scF)>0$ by the $\xi$-stability of $\scE$.
Thus $\Phi(\scE)$ is $\theta$-stable.

Conversely, suppose $E$ is a $\theta$-stable $G$-constellation on $V$.
Then it is $\theta^N$-semistable by our choice of $m$
and therefore $\scE:=\Phi^{-1}(E)$ is a $\Gbar$-constellation by Proposition \ref{prop:bijection}.
For a subsheaf $\scF \subset \scE$,
$F:=\Phi(\scF)$ has a filtration as an object of $\coh^N(V)$
whose factors are $N$-constellations.
Therefore $F$ satisfies $\theta^N(F)=0$
and hence we obtain $\chi(\xi, \scF)=\theta(F)>0$, which proves the $\xi$-stability of $\scF$.

Thus we have a bijection between $\xi$-stable $\Gbar$-constellations and $\theta$-stable $G$-constellations.
To establish an isomorphism $\scM_\theta(V) \cong \scM_\xi(Y_N)$,
we show that for any locally noetherian scheme $S$ over $\bC$,
this bijection can be extended to a bijection between flat families of $\xi$-stable $\Gbar$-constellations
and flat families of $\theta$-stable $G$-constellations parameterized by $S$.
Let $\cU$ be the universal $N$-constellation on $Y_N \times V$ and $\scU_S$ be the pull back of $\scU$ to $Y_N \times V \times S$.
Then we can define a functor
\[
\Phi_S: D^b(\coh^{\Gbar}Y_N \times S) \to D^b(\coh^{G} V \times S)
\]
by
\[
\Phi_S(-)=\bR(p_{V\times S})_*(\cU_S \otimes p_{Y_N\times S}^*(-))
\]
whose quasi-inverse is given by
\[
\Phi_S^{-1}(-)=((p_{Y_N\times S})_*(\cU_S^\vee[\dim V] \overset{\bL}{\otimes} p_{V\times S}^*(-))^N.
\]
Suppose $\cE_S$ is a flat family of $\xi$-stable
$\Gbar$-constellations on $Y_N$ parameterized by $S$.
Then, for any geometric point $s$ of $S$, we have $\Phi_S(\cE_S)\overset{\bL}{\otimes} \cO_s \cong \Phi(\cE_s)$ as in \cite[Lemma 4.1]{Bridgeland_ETCFMT}, which is a $\theta$-stable $G$-constellation on $V$.
Hence the argument in \cite[Proposition 4.2]{Bridgeland_ETCFMT} implies that $\Phi_S(\cE_S)$ is actually a flat family of $G$-constellations on $V$.
Conversely, if $E_S$ is a flat family of $\theta$-stable $G$-constellations,
the same argument shows that $\Phi_S^{-1}(E_S)$ is a flat family of $\xi$-stable
$N$-constellations on $Y_N$.
\end{proof}
%\begin{remark}
%The above theorem shows that $\cM_\theta$, together with the transform $\Phi_{\cM_\theta}^{-1}(E_{\cM_\theta})$ of the universal family $E_{\cM_\theta}$ of $\theta$-stable $G$-constellations on $V$,
%represents the moduli functor for $\xi$-stable $\Gbar$-constellations on $Y_N$
%and this proves the existence of the moduli space in such a case.
%\end{remark}
%%%%%%%%%%%%%
\section{The case $G \ni -I$}
In this section, put $V=\bC^2$ and
assume that $G\subset\GL(V)$ contains $-I$, where $I$ is the identity matrix.
We put $N:=\langle -I \rangle \subset G$ and $\Gbar:=G/N$.
Let $\theta^N$ be any generic stability parameter for $N$-constellations
(which is automatically fixed by the conjugate action of $G$ since $N$ is central)
and let $Y_N=\scM_{\theta^N}(V)$ be the moduli space of $N$-constellations on $V$,
on which $\Gbar$ acts naturally.
Since $Y_N$ is a crepant resolution of the $A_1$ singularity $V/N$,
the maximal resolution of $(Y_N/\Gbar, B_N)$ coincides with the maximal resolution
of $(X, B)$, where $B_N$ is the boundary divisor on $Y_N$ determined by the ramification of
$Y_N \to Y_N/\Gbar$.

%Since $N \subset \SL(2, \bC)$, there is an equivalence \cite{iu}
%$$
%D^b(\coh^{\Gbar}(Y_N)) \cong D^b(\coh^G(\bC^2)).
%$$
Let $C$ be the exceptional curve of $Y_N \to V/N$.
Then the equivalence \eqref{eq:equivalence} restricts to the equivalence
\begin{equation}\label{eq:cptspt}
\Phi:D^b(\coh^{\Gbar}_C(Y_N)) \cong D^b(\coh^G_{0}(V))
\end{equation}
of full subcategories consisting of objects supported by the subsets $C\subset Y_N$ and $\{0\}\subset V$ respectively.
Consider the Grothendieck groups of \eqref{eq:cptspt}:
\begin{equation}\label{eq:Kcptspt}
K(\coh^{\Gbar}_C(Y_N)) \cong K(\coh^G_{0}(V)),
\end{equation}
where $K(\coh^G_{0}(V))$ is isomorphic to the representation ring $R(G)$ of $G$.
Recall that there is a perfect pairing
$$
\chi: K(\coh^G(V)) \times K(\coh^G_{0}(V)) \to \bZ
$$
defined by \eqref{eq:pairing},
which is isomorphic to
$$
\chi: K(\coh^{\Gbar}(Y_N)) \times K(\coh^{\Gbar}_C(Y_N)) \to \bZ
$$
by $\Phi$.
%The Grothendieck group $K(\coh^{\Gbar}_C(Y_N))$ has a filtration by dimensions of supports.
Let
$$F_iK(\coh^{\Gbar}_C(Y_N))\subset K(\coh^{\Gbar}_C(Y_N))$$
be the subgroup generated by the classes of objects
whose supports are at most $i$-dimensional.
%and let
%$$
%F^iK(\coh^{\Gbar}(Y_N)) \subset K(\coh^{\Gbar}(Y_N))
%$$
%be the subgroup consisting of the classes of objects
%whose supports are at least of codimension $i$.
%As in \cite[Proposition 5.1]{craw-ishii}, we have a perfect pairing
%$$
%\chi:K(\coh^{\Gbar}(Y_N))/F^1K(\coh^{\Gbar}(Y_N)) \times F_0K(\coh^{\Gbar}_C(Y_N)) \to \bZ.
%$$
Then the classes of $\Gbar$-constellations on $Y_N$ lie in $F_0K(\coh^{\Gbar}_C(Y_N))$
and for a stability parameter
$$\xi \in K(\coh^{\Gbar}(Y_N))_\bQ \cong K(\coh^{\Gbar}_C(Y_N))_\bQ^*,$$
the actual stability condition depends only on its image in $F_0K(\coh^{\Gbar}_C(Y_N))^*_\bQ$.
In the next subsection, we investigate the structure of $F_0K(\coh^{\Gbar}_C(Y_N))$.

\subsection{Structure of $F_0K(\coh^{\Gbar}_C(Y_N))$}\label{subsec:key}
In this subsection,
we assume that $G$ is not abelian.
Notice that $G$ acts on the exceptional curve $C \cong \bP(V)$
through the homomorphism
$$G \hookrightarrow \GL(V) \twoheadrightarrow \PGL(V)$$
and let $Z \subset G$ be the kernel of $G \to \PGL(V)$.
It is the subgroup consisting of scalar matrices in $G$.

Since $G$ is non-abelian,
$G/Z \subset \PGL(V)$ is a polyhedral (or dihedral) group
acting on $\bP(V)$ which we regard as a (real) $2$-sphere.
There are three non-free $G/Z$-orbits in $C$:
the projections of the vertices, edges and faces of the regular polyhedron
to the sphere.
These orbits are denoted by $O_1$, $O_2$ and $O_3$ respectively.

For a $\Gbar$-orbit $O \subset C$,
let $\coh^{\Gbar}_O(Y_N)$ denote the category of $\Gbar$-equivariant coherent sheaves
supported on $O$.
Then we have an equivalence
\begin{equation}\label{eq:stabilizer_cat}
\coh^{\Gbar}_O(Y_N) \cong \coh^{\Gbar_P}_P(Y_N)
\end{equation}
where $\Gbar_P$ is the stabilizer subgroup of a point $P\in O$
and $\coh^{\Gbar_P}_P(Y_N)$ is the category of $\Gbar_P$-equivariant coherent sheaves
supported on $P$.
Taking the Grothendieck groups of the both sides, we obtain
\begin{equation}\label{eq:stabilizer}
K(\coh^{\Gbar}_O(Y_N)) \cong R(\Gbar_P)
\end{equation}
where $R(\Gbar_P)$ is the representation ring of $\Gbar_P$ regarded as an additive group.

Let $\Gbar_k\subset \Gbar$ be the stabilizer subgroup of a point in $O_k$,
which is an abelian group since $\Zbar:=Z/N \subset \Gbar_k$ is central and
$\Gbar_k/\Zbar$ is cyclic.
We consider the pushforward maps
\begin{equation} \label{eq:push_forward}
K(\coh^{\Gbar}_{O_k}(Y_N)) \to F_0K(\coh^{\Gbar}_C(Y_N))
\end{equation}
for $k=1,2,3$.
By \eqref{eq:stabilizer} for $O=O_k$, these maps are regarded as maps
$$
\beta_k: R(\Gbar_k) \to F_0K(\coh^{\Gbar}_C(Y_N)).
$$
Since $\Zbar$ is a subgroup of $\Gbar_k$, we have the induction maps
$$
\alpha_k: R(\Zbar) \to R(\Gbar_k).
$$
Define a map $\alpha : R(\Zbar)^{\oplus 2} \to R(\Gbar_1) \oplus R(\Gbar_2) \oplus R(\Gbar_3)$ by
$$
\alpha(a, b)=(\alpha_1(a), -\alpha_2(a)+\alpha_2(b), -\alpha_3(b)).
$$
The purpose of this subsection is to prove the following.
\begin{proposition}
Let $\Gbar_k$, $\beta_k$, $\alpha$ be as above.
Then the following is an exact sequence of additive groups:
$$
0 \to R(\Zbar)^{\oplus 2} \overset{\alpha}{\to} R(\Gbar_1) \oplus R(\Gbar_2) \oplus R(\Gbar_3)
\overset{\beta}{\to} F_0K(\coh^{\Gbar}_C(Y_N)) \to 0
$$
where $\beta=(\beta_1, \beta_2, \beta_3)$.
\end{proposition}
The proof of the proposition is divided into three steps below.
We first show that $\beta$ is surjective:
\begin{step}\label{step:non-generic_stabilizer}
The additive group
$F_0K(\coh^{\Gbar}_C(Y_N))$ is generated by sheaves supported on
$O_1 \cup O_2 \cup O_3$.
\end{step}
\begin{proof}
It is obvious that $F_0K(\coh^{\Gbar}_C(Y_N))$ is generated by simple objects
(objects having no non-trivial subobjects).
Moreover, a simple object is supported on a single orbit $O$ and is determined by an
irreducible representation of the stabilizer subgroup $\Gbar_P$ of a point $P \in O$
by \eqref{eq:stabilizer_cat}.
Therefore, it is sufficient to show that the class in $K(\coh^{\Gbar}_C(Y_N))$ of a simple object $\scE$ supported on a free $G/Z$-orbit $O_f$ coincides with the class of some object $\scF$ supported on $O_1 \cup O_2 \cup O_3$.
Actually, we prove that for any $k \in \{1,2,3\}$ we can choose such an object $\scF$ supported on $O_k$.
Simple objects supported on the orbit $O_f$ are determined by irreducible representations
of the stabilizer subgroup $\Zbar \subset \Gbar$ by \eqref{eq:stabilizer_cat}.
To describe them, notice that $C=\bP(V)$ carries a $G$-equivariant line bundle $\scL=\scO_C(1)$
on which an element $\lambda I \in Z$ acts as the fiber-wise scalar multiplication by $\lambda$.
On $\scL^2$, the $G$-action is reduced to a $\Gbar$-action and the
induced action of $\Zbar$ on
the fibers of $\scL^0, \scL^2, \dots, \scL^{2(l-1)}$ are the irreducible representations of the cyclic group $\Zbar$, where $l$ is the order of $\Zbar$.
Therefore, the simple objects supported on $O_f$ are
\begin{equation}\label{eq:simple_objects}
\scL^0|_{O_f}, \scL^2|_{O_f}, \dots, \scL^{2(l-1)}|_{O_f},
\end{equation}
where we regard ${O_f}$ as a reduced subscheme.
Now consider the exact sequences
$$
0 \to \scL^{2i}\otimes \scO_C(-{O_f}) \to \scL^{2i} \to \scL^{2i}|_{O_f} \to 0
$$
and 
$$
0 \to \scL^{2i}\otimes \scO_C(-n_kO_k) \to \scL^{2i} \to \scL^{2i}|_{n_kO_k} \to 0
$$
for any $k \in \{1, 2, 3\}$ where $n_k$ is the order of $\Gbar_k/\Zbar$.
If we show $\scO_C(-{O_f}) \cong \scO_C(-n_kO_k)$ in $\coh^{\Gbar}(Y_N)$, then we obtain
\begin{equation}\label{eq:orbit_change}
[\scL^{2i}|_{O_f}]=[\scL^{2i}|_{n_kO_k}]
\end{equation}
in $K(\coh^{\Gbar}_C(Y_N))$ for any $k$ as desired.

Finally, we show $\scO_C(-{O_f}) \cong \scO_C(-n_kO_k)$.
Let $\Cbar \cong \bP^1$ be the quotient of $C$ by the action of $G/Z$.
%the action of $G$ on $\Cbar$ is trivial.
Then both $\scO_C(-{O_f})$ and $\scO_C(-n_kO_k)$ are the pull-backs of
$\scO_{\Cbar}(-1)$ (equipped with the trivial $\Gbar$-action) and hence we obtain the isomorphism.
\end{proof}
%%%

\begin{step}
$\beta\circ \alpha=0$.
\end{step}
\begin{proof}
This is equivalent to the equality
$$
\beta_1\circ\alpha_1=\beta_2\circ\alpha_2=\beta_3\circ\alpha_3.
$$
We recall the isomorphism \eqref{eq:stabilizer} for  a free $G/Z$-orbit $O_f \subset C$:
$$
R(\Zbar) \cong K(\coh^{\Gbar}_{O_f}(Y_N)).
$$
Then it is sufficient to prove that $\beta_k\circ\alpha_k$ is identified with the pushforward map
$$
K(\coh^{\Gbar}_{O_f}(Y_N)) \to F_0K(\coh^{\Gbar}_C(Y_N)).
$$
Recall that $K(\coh^{\Gbar}_{O_f}(Y_N))$ has a basis of the form \eqref{eq:simple_objects}
and that their images in $K(\coh^{\Gbar}_C(Y_N))$ satisfy \eqref{eq:orbit_change}.
Hence the problem is reduced to showing that the map
$$K(\coh^{\Gbar}_{O_f}(Y_N)) \to K(\coh^{\Gbar}_{O_k}(Y_N))$$ defined by
$$[\scL^{2i}|_{O_f}] \mapsto [\scL^{2i}|_{n_kO_k}]$$
is identified with the induction map $\alpha_k$.
The irreducible representation $\rho_i$ of $\Zbar$ corresponding to $[\scL^{2i}|_{O_f}]$
is defined by sending $[\lambda I] \in \Zbar$ to $\lambda^{2i} \in \bC^\times$.
On the other hand, we have
$$
[\scL^{2i}|_{n_kO_k}] = \sum_{j=0}^{n_k-1} [\scL^{2i}(-j O_k)|_{O_k}].
$$
Here $\scL^{2i}|_{O_k}$ corresponds to a representation of $\Gbar_k$
whose restriction to $\Zbar$ is $\rho_i$.
Moreover, $\scO_C(-jO_k)|_{O_k}$ ($0 \le j \le n_k-1$) correspond to the
irreducible representations of the cyclic group $\Gbar_k/\Zbar$.
Thus the element of $R(\Gbar_k)$ corresponding to $[\scL^{2i}|_{n_kO_k}]$ is the sum
of all the irreducible representations of $\Gbar_k$ whose restrictions to $\Zbar$ are $\rho_i$.
Since $\Gbar_k$ is an abelian group, this is the induced representation of $\rho_i$.
Thus we obtain $\beta \circ \alpha=0$.
\end{proof}

\begin{step}
$\ker \beta = \Image \alpha$.
\end{step}
\begin{proof}
Notice that $\coker \alpha$ is torsion free, $\beta$ is surjective and $\beta \circ \alpha=0$.
Therefore it suffices to show
$$
\rank F_0K(\coh^{\Gbar}_C(Y_N)) = \sum_{k=1}^3 \rank R(\Gbar_k) - 2 \rank R(\Zbar).
$$
This follows from the following two equalities:
\begin{align}
\rank F_0K(\coh^{\Gbar}_C(Y_N)) &= \rank R(G) - \rank R(\Zbar) \label{eq:rankF0} \\
\sum_{k=1}^3 \rank R(\Gbar_k) &= \rank R(G) + \rank R(\Zbar).\label{eq:ranksum}
\end{align}
We first consider \eqref{eq:rankF0}.
%Notice that $\Phi$ induces an isomorphism
%$$
%K(\coh^{\Gbar}_C(Y_N)) \cong K(\coh^G_0(V)) \cong R(G)
%$$
%where $\coh^G_0(V)$ is the subcategory of $\coh^G(V)$ consisting of objects
%supported on the origin $0 \in V$.
The isomorphism \eqref{eq:Kcptspt} reduces  \eqref{eq:rankF0} to the equality
$$
\rank K(\coh^{\Gbar}_C(Y_N))/F_0K(\coh^{\Gbar}_C(Y_N)) = \rank R(\Zbar)
$$
and therefore it suffices to show that the classes
\begin{equation}\label{eq:basis}
[\scO_C], [\scL^{2}], \dots, [\scL^{2(l-1)}]
\end{equation}
form a free basis of the quotient $K(\coh^{\Gbar}_C(Y_N))/F_0K(\coh^{\Gbar}_C(Y_N))$
where
\[
l:=\rank R(\Zbar)=|\Zbar|.
\]
Recall that $\cL^2\cong \omega_C^{-1}$ is a $\Gbar$-equivariant line bundle on $C=\bP(V)$.
Since $\Zbar$ acts on $C$ trivially, if we regard $\cL^2$ as an object of $\coh^{\Zbar}(C)$,  we have
 \begin{equation}\label{eq:L^2i}
 \cL^{2i} \cong \cO_C(2i) \otimes \rho_{\ibar}\quad \text{in $\coh^{\Zbar}(C)$}
 \end{equation}
where $\ibar= i \mod l$ and
$\rho_0, \rho_1, \dots, \rho_{l-1}$
are the irreducible representations of the cyclic group $\Zbar \cong \bZ/l\bZ$.
This implies that \eqref{eq:basis} is linearly independent.
To see that \eqref{eq:basis} is a generator, we show that for any object $\scE \in \coh^{\Gbar}_C(Y_N)$
its class $[\scE]$ is a linear combination of \eqref{eq:basis} modulo $F_0K(\coh^{\Gbar}_C(Y_N))$.
We may assume that $\scE$ is a locally free sheaf on $C$ and we use the induction on $\rank \scE$.
If $\rank \scE=0$, there is nothing to prove and we may suppose $\rank \scE >0$.
If we regard $\scE$ as an object of $\coh^{\Zbar}(C)$, it splits as
$\scE=\bigoplus_i \scE_i \otimes_{\bC} \rho_i$ with $\scE_i \in \coh(C)$.
Suppose $\scE_i \ne 0$.
For any integer $m$ we have
\begin{equation}\label{eq:ZG/Z}
\Hom_{\scO_C}(\scL^{2i}, \scE\otimes \scL^{2lm})^{\Gbar}
=H^0((\scE \otimes \scL^{2ml-2i})^{\Zbar})^{\Gbar/\Zbar}.
\end{equation}
Here, \eqref{eq:L^2i} shows
$$(\scE \otimes \scL^{2ml-2i})^{\Zbar} \cong \scE_i \otimes \scO(2ml-2i) \ne 0$$
and the restriction map
$$
H^0((\scE \otimes \scL^{2ml-2i})^{\Zbar}) \to H^0((\scE \otimes \scL^{2ml-2i})^{\Zbar}|_{O_f})
$$
is surjective for a $\Gbar/\Zbar$-free orbit $O_f\subset C$ if $m$ is sufficiently large.
Since $H^0((\scE \otimes \scL^{2ml+2i})^{\Zbar}|_{O_f})$ is a non-zero multiple of the regular representation of
$\Gbar/\Zbar$, its $\Gbar/\Zbar$-invariant part is non-zero.
Therefore, \eqref{eq:ZG/Z} is non-zero and hence there is a non-zero homomorphism
$$\alpha: \scL^{2i} \hookrightarrow \scE\otimes \scL^{2lm}.$$
Now the induction hypothesis shows that $\coker \alpha$ is a linear combination of \eqref{eq:basis} modulo $F_0K(\coh^{\Gbar}_C(Y_N))$.
This shows that the class $[\scE\otimes \scL^{2lm}]$ is also a linear combination of \eqref{eq:basis} modulo $F_0K(\coh^{\Gbar}_C(Y_N))$.
Since we have
$$[\scE]-[\scE\otimes \scL^{2lm}] \in F_0K(\coh^{\Gbar}_C(Y_N)),$$
$[\scE]$ is a linear combination of \eqref{eq:basis} modulo $F_0K(\coh^{\Gbar}_C(Y_N))$.
Thus \eqref{eq:basis} is a free basis of $K(\coh^{\Gbar}_C(Y_N))/F_0K(\coh^{\Gbar}_C(Y_N))$
and therefore we have established \eqref{eq:rankF0}.

Next we prove \eqref{eq:ranksum}.
Let $\ZL(V) \subset \GL(V)$ be the subgroup consisting of the non-zero scalar matrices and consider
the multiplication map
$$
\mu: ZL(V) \times \SL(V) \to \GL(V).
$$
Then the kernel of $\mu$ is a group of order $2$ generated by $(-I, -I)$.
We put $\Gtilde=\mu^{-1}(G)$ and let $H\subset \SL(V)$ be the image of $\Gtilde$
with respect to the second projection.
For any element $(z, h) \in \Gtilde$, denote by $Z_{\Gtilde}(z,h)$ and $Z_G(zh)$
the centralizers of $(z, h)$ in $\Gtilde$ and $zh$ in $G$ respectively.
Then the restriction $\mu: Z_{\Gtilde}(z,h) \to Z_G(zh)$ is a surjective two-to-one map
and hence the number of conjugates of $(z,h)$ coincides with the number of
conjugates of $zh$.
Therefore, the number of conjugacy classes in $\Gtilde$ is twice the number
of conjugacy classes in $G$.
Thus we obtain
\begin{equation*}
\rank R(G)= \frac{1}{2}\rank R(\Gtilde).
\end{equation*}
Moreover, since $\Gtilde/Z \cong H$ and $Z$ is central in $\Gtilde$,
this can be written as
\begin{equation}\label{eq:rankR(G)}
\rank R(G) =  \frac{1}{2} \rank R(H) \times |Z|=\rank R(H) \times |\Zbar|. 
\end{equation}
Notice that $H$ acts on $V$ and $\Hbar:=H/N \cong \Gbar/\Zbar\subset \PGL(V)$ acts on $C=\bP(V)$.
Since $H$ is in $\SL(V)$, the McKay correspondence for the binary polyhedral (or dihedral) group $H$ establishes
\begin{equation}\label{eq:rankR(H)}
\sum_{k=1}^3 |\Hbar_k|=\rank R(H) +1
\end{equation}
where $\Hbar_k \subset \Hbar$ is the stabilizer of a point in $O_k$
(the left hand side of \eqref{eq:rankR(H)} is two plus the number of the irreducible exceptional curves
in the minimal resolution of $V/H$, which is also the minimal resolution of $Y_N/\Hbar$).
Moreover, the isomorphism
$\Hbar \cong \Gbar/\Zbar$
implies
\begin{equation}\label{eq:rankR(G_k)}
|\Hbar_k| \times |\Zbar| = |\Gbar_k| = \rank R(\Gbar_k).
\end{equation}
Putting the equalities \eqref{eq:rankR(G)}, \eqref{eq:rankR(H)} and \eqref{eq:rankR(G_k)} together, we obtain \eqref{eq:ranksum}.
\end{proof}

\begin{corollary}\label{cor:F_0^*}
The dual module $\Hom_\bZ(F_0K(\coh_C^{\Gbar}(Y_N)), \bZ)$ is isomorphic to
$$
\left\{(\theta_1, \theta_2, \theta_3) \in \bigoplus_{k=1}^3 \Hom_\bZ(R(\Gbar_k), \bZ)
\;\middle|\; \theta_1|_{\Zbar}=\theta_2|_{\Zbar}=\theta_3|_{\Zbar}\right\}.
$$
\end{corollary}
\subsection{Main theorem}\label{subsec:main}
\begin{proposition}\label{prop:-I}
Suppose a finite subgroup $G\subset \GL(2, \bC)$ contains $-I$ and $Y \to Y_N/\Gbar$ is a resolution
dominated by $\Ymax$.
Then there exists a generic stability parameter $\theta \in \Theta$
such that $\scM_\theta \cong Y$.
Especially, the maximal resolution $\Ymax$ of $(\bC^2/G, B)$
is isomorphic to the moduli space of $G$-constellations for some generic stability parameter $\theta$.
\end{proposition}
\begin{proof}
We may assume $G$ is non-abelian by Theorem \ref{thm:abelian}
so we may apply the results of section \ref{subsec:key}.
%Put $N=\langle -I \rangle$ and let $\theta^N$ be any generic stability parameter for $N$-clusters
%(which is automatically fixed by the conjugate action of $G$ since $N$ is central).
If we show there exists a generic parameter $\xi \in K(\coh^{\Gbar}(Y_N))^0_\bQ$ such that
$\scM_\xi(Y_N) \cong Y$,
then the assertion follows from Theorem \ref{thm:iterated_construction}.

Let $P \in C$ be a point.
Since $\Gbar$ acts on $Y_N \times \bC=\cM_{\theta^N}(V \times \bC)$ and $\Zbar$ fixes $(P,0)$,
$\Zbar$ acts on the Zariski tangent space $\Ttilde:=T_{(P,0)}(Y_N \times \bC) \cong \bC^3$
as a subgroup of $\SL(\Ttilde)$.
Note that as a representation of $\Zbar$, $\Ttilde$ is independent of the choice of the point $P$.
Let $T' \subset \Ttilde$ be the two-dimensional $\Zbar$-invariant subspace transversal to $C$;
then $\Zbar \subset \SL(T')$.
Fix a generic stability parameter $\theta^{\Zbar} \in R(\Zbar)^*_\bQ$
for $\Zbar$-constellations (on $\Ttilde$) satisfying $\theta^{\Zbar}(\bC[\Zbar])=0$.
Then $W:=\scM_{\theta^{\Zbar}}(T')$ is the minimal resolution
of $T'/\Zbar$.
The Fourier-Mukai transform
$$
\varphi_{\theta^{\Zbar}}^*: R(\Zbar)^*_\bQ \cong K(\coh^{\Zbar}(T'))_\bQ \overset{\sim}{\longrightarrow}
K(\coh W)_\bQ
$$
sends
$\theta^{\Zbar}$ to an element $l_{\theta^{\Zbar}}$ of $F^1K(\coh W)_\bQ \cong \Pic(W)_\bQ$
and it lies in the ample cone $\Amp(W)$ as in \eqref{eq:FM(theta)}.
(Notice that here $\dim T'=2$ and $F^2K(\coh W)=0$.)

Take a point $P_k$ in the orbit $O_k$ for each $k \in \{1,2,3\}$.
We consider the tangent spaces $\Ttilde_k:=T_{(P_k,0)}(Y_N \times \bC)$ and $T_k=T_{P_k}(Y_N)$.
Let $R_k$ denote the complete local ring of $T_k/\Gbar_k$ at $[0]$
which is isomorphic to the complete
local ring of $Y_N/\Gbar$ at $[P_k]$:
\[
R_k:=\widehat{\cO}_{T_k/\Gbar_k, [0]} \cong \widehat{\cO}_{Y_N/\Gbar, [P_k]}.
\]
By this isomorphism, there is a resolution
$$Y_k \to T_k/\Gbar_k$$
with an isomorphism
\begin{equation}\label{eq:formal_isom}
Y_k \times_{(T_k/\Gbar_k)} \Spec R_k \cong Y \times_{(Y_N/\Gbar)} \Spec R_k
\end{equation}
over $\Spec R_k$.
Since $\Gbar_k$ is abelian, we can apply Proposition \ref{prop:ample} where
the first factor of $T_k \cong \bC^2$ is $T_{P_k}(C)$
(so that $(1,0)$ lies in $T_{P_k}(C)$ and $G_{(1,0)}=\Zbar$)
and obtain a projective crepant resolution
$$
U_{\Sigma_k} \to \Ttilde_k/{\Gbar_k}
$$
such that $Y_k \subset U_{\Sigma_k}$ and that the restriction map $\Amp(U_{\Sigma_k}) \to \Amp(W)$
is surjective.
Choose a class $l_k \in \Amp(U_{\Sigma_k})$ which is mapped to $l_{\theta^{\Zbar}} \in \Amp(W)$ for each $k$.
Then by Theorem \ref{thm:craw-ishii} we can find a generic stability parameter
$\theta_k$ for $\Gbar_k$-constellations on $\Ttilde_k$ such that
$\scM_{\theta_k}(\Ttilde_k) \cong U_{\Sigma_k}$ and the class of $\varphi_{\theta_k}^*(\theta_k)$ in $\Pic(U_{\Sigma_k})_\bQ$ coincides with $l_k$.
Since $[\varphi_{\theta_k}^*(\theta_k)]=l_k$ and $l_k$ restricts to $l_{\theta^\Zbar}$,
$\theta_k$ restricts to $\theta^{\Zbar}$ on $R(\Zbar)$.
Then  Corollary \ref{cor:F_0^*} shows that $(\theta_1, \theta_2, \theta_3)$ determines
an element of $F_0K(\coh_C^{\Gbar}(Y_N))^*_\bQ$.
Lift it to an element $\xi \in K(\coh^{\Gbar}(Y_N))_\bQ \cong K(\coh_C^{\Gbar}(Y_N))^*_\bQ$.
Since the restriction of $\xi$ to $K(\coh^{\Gbar}(O_k))_\bQ \cong R(\Gbar_k)^*_\bQ$ is $\theta_k$ which is of rank $0$,
we have $\rank \xi=0$
%%%%
% rank 0 condition
%%%%
and we can consider the moduli space $\scM_\xi(Y_N)$.

We claim that there is an isomorphism
\begin{align}\label{eq:isom_over_R_k}
\cM_\xi(Y_N) \times_{(Y_N/\Gbar)} \Spec R_k \cong \cM_{\theta_k}(T_k) \times_{(T_k/\Gbar_k)} \Spec R_k
\end{align}
over $\Spec R_k$.
For any locally noetherian scheme $S$ over $\Spec R_k$, 
an $S$-valued point of the left hand side of \eqref{eq:isom_over_R_k} is
given by a flat family of $\xi$-stable $\Gbar$-constellations on $Y_k$ parameterized by $S$,
which is an object of $\coh^{\Gbar}(Y_N \times_{(Y_N/\Gbar)} S)$.
Similarly, an $S$-valued point of the right hand side of \eqref{eq:isom_over_R_k} is
given by a flat family of $\theta_k$-stable $\Gbar_k$-constellations on $T_k$ parameterized by $S$,
which is an object of $\coh^{\Gbar_k}(T_k \times_{(T_k/\Gbar_k)} S)$.

Notice that
\[
\begin{aligned}
Y_N \times_{(Y_N/\Gbar)} S & \cong (Y_N \times_{(Y_N/\Gbar)} \Spec R_k)\times_{(\Spec R_k)} S \\
					&\cong \left(\coprod_{Q \in O_k} \Spec \widehat{\cO}_{Y_N, Q}\right) \times_{(\Spec R_k)} S \\
					&\supset \Spec \widehat{\cO}_{Y_N, P_k} \times_{(\Spec R_k)} S \\
					&\cong \Spec \widehat{\cO}_{T_k, 0} \times_{(\Spec R_k)} S \\
					&\cong T_k \times_{(T_k/\Gbar_k)} S
					\end{aligned}
\]
which induces an equivalence
\[
\coh^{\Gbar}(Y_N \times_{(Y_N/\Gbar)} S ) \cong \coh^{\Gbar_k}(T_k \times_{(T_k/\Gbar_k)} S).
\]
(this is almost the same as \eqref{eq:stabilizer_cat}).
This equivalence gives a bijection between $S$-valued points of the both sides of \eqref{eq:isom_over_R_k}
and we obtain \eqref{eq:isom_over_R_k}.

Our choice of $\theta_k$ implies $\cM_{\theta_k}(T_k) \cong Y_k$ and hence  \eqref{eq:formal_isom} and  \eqref{eq:isom_over_R_k} yield an isomorphism
\[
\cM_\xi(Y_N) \times_{(Y_N/\Gbar)} \Spec R_k \cong Y \times_{(Y_N/\Gbar)} \Spec R_k.
\]
over $\Spec R_k$.
Since $\cM_\xi(Y_N)$ and $Y$ are both isomorphic to $Y_N/\Gbar$ except over the points $[P_1]$, $[P_2]$, and $[P_3]$,
we obtain $\cM_\xi(Y_N) \cong Y$.
\end{proof}
 
Recall that we say $G \subset \GL(2, \bC)$ is {\it small} if $G$ acts freely on $\bC^2 \setminus \{0\}$.
The following lemma follows from the classification of small subgroups of $\GL(2, \bC)$
but we give a proof for the reader's sake.
\begin{lemma}
If a finite small subgroup $G \subset \GL(2, \bC)$ is non-abelian, then it contains $-I$
as a unique element of order $2$. 
\end{lemma}
\begin{proof}
If $G$ is non-abelian, then its image $G' \subset \PGL(2, \bC)$ is also non-abelian
and therefore it is either a dihedral or a polyhedral group.
Especially, the orders $|G'|$ and  $|G|$ are even.
Then $G$ contains an element of order $2$.
If it is not $-I$, then it fixes a line in $\bC^2$,
contradicting the smallness of $G$.
\end{proof}
\begin{theorem}\label{thm:small}
If $G \subset \GL(2, \bC)$ is a finite small subgroup,
then Conjecture \ref{conjecture:main} is true.
\end{theorem}
\begin{proof}
The abelian case follows from Theorem \ref{thm:abelian}.
Otherwise, $G$ contains $-I$ by the above lemma.
Moreover, the minimal resolution of $V/G$ factors through $Y_N/\Gbar$; see \cite{Brieskorn_RSKF}.
Then the assertion follows from Proposition \ref{prop:-I}.
\end{proof}

\bibliographystyle{amsalpha}
%\bibliography{bibs}
\def\cprime{$'$} \def\cprime{$'$}
\providecommand{\bysame}{\leavevmode\hbox to3em{\hrulefill}\thinspace}
\providecommand{\MR}{\relax\ifhmode\unskip\space\fi MR }
\providecommand{\MRhref}[2]{%
  \href{http://www.ams.org/mathscinet-getitem?mr=#1}{#2}
}
\providecommand{\href}[2]{#2}

\noindent
Akira Ishii

Graduate School of Mathematics, Nagoya University,
Furocho, Chikusaku, Nagoya, 464-8602
Japan

{\em e-mail address}\ : \ akira141@math.nagoya-u.ac.jp

%\ \\

\end{document}